\documentclass[12pt,oneside]{amsart}

\usepackage{amssymb}
\usepackage{amsfonts}
\usepackage{amscd}
\usepackage[matrix, arrow, curve]{xy}
\usepackage{graphicx}

\numberwithin{equation}{section}
\newtheorem{theorem}{Theorem}[section]
\newtheorem*{Thm}{Theorem}

\newtheorem{proposition}[theorem]{Proposition}
\newtheorem{corollary}[theorem]{Corollary}
\newtheorem{lemma}[theorem]{Lemma}
\theoremstyle{definition}
\newtheorem{definition}[theorem]{Definition}
\theoremstyle{remark}
\newtheorem*{remark}{Remark}
\newtheorem{example}{Example}

\begin{document}
\newcommand{\M}{\mathcal{M}}
\newcommand{\F}{\mathcal{F}}

\newcommand{\Teich}{\mathcal{T}_{g,N+1}^{(1)}}
\newcommand{\T}{\mathrm{T}}
\newcommand{\corr}{\bf}
\newcommand{\vac}{|0\rangle}
\newcommand{\Ga}{\Gamma}
\newcommand{\new}{\bf}
\newcommand{\define}{\def}
\newcommand{\redefine}{\def}
\newcommand{\Cal}[1]{\mathcal{#1}}
\renewcommand{\frak}[1]{\mathfrak{{#1}}}
\newcommand{\Hom}{\rm{Hom}\,}
\newcommand{\refE}[1]{(\ref{E:#1})}
\newcommand{\refCh}[1]{Chapter~\ref{Ch:#1}}
\newcommand{\refS}[1]{Section~\ref{S:#1}}
\newcommand{\refSS}[1]{Section~\ref{SS:#1}}
\newcommand{\refT}[1]{Theorem~\ref{T:#1}}
\newcommand{\refO}[1]{Observation~\ref{O:#1}}
\newcommand{\refP}[1]{Proposition~\ref{P:#1}}
\newcommand{\refD}[1]{Definition~\ref{D:#1}}
\newcommand{\refC}[1]{Corollary~\ref{C:#1}}
\newcommand{\refL}[1]{Lemma~\ref{L:#1}}
\newcommand{\R}{\ensuremath{\mathbb{R}}}
\newcommand{\C}{\ensuremath{\mathbb{C}}}
\newcommand{\N}{\ensuremath{\mathbb{N}}}
\newcommand{\Q}{\ensuremath{\mathbb{Q}}}
\renewcommand{\P}{\ensuremath{\mathcal{P}}}
\newcommand{\Z}{\ensuremath{\mathbb{Z}}}
\newcommand{\kv}{{k^{\vee}}}
\renewcommand{\l}{\lambda}
\newcommand{\gb}{\overline{\mathfrak{g}}}
\newcommand{\dt}{\tilde d}     
\newcommand{\hb}{\overline{\mathfrak{h}}}
\newcommand{\g}{\mathfrak{g}}
\newcommand{\h}{\mathfrak{h}}
\newcommand{\gh}{\widehat{\mathfrak{g}}}
\newcommand{\ghN}{\widehat{\mathfrak{g}_{(N)}}}
\newcommand{\gbN}{\overline{\mathfrak{g}_{(N)}}}
\newcommand{\tr}{\mathrm{tr}}
\newcommand{\gln}{\mathfrak{gl}(n)}
\newcommand{\son}{\mathfrak{so}(n)}
\newcommand{\spnn}{\mathfrak{sp}(2n)}
\newcommand{\sln}{\mathfrak{sl}}
\newcommand{\sn}{\mathfrak{s}}
\newcommand{\so}{\mathfrak{so}}
\newcommand{\spn}{\mathfrak{sp}}
\newcommand{\tsp}{\mathfrak{tsp}(2n)}
\newcommand{\gl}{\mathfrak{gl}}
\newcommand{\slnb}{{\overline{\mathfrak{sl}}}}
\newcommand{\snb}{{\overline{\mathfrak{s}}}}
\newcommand{\sob}{{\overline{\mathfrak{so}}}}
\newcommand{\spnb}{{\overline{\mathfrak{sp}}}}
\newcommand{\glb}{{\overline{\mathfrak{gl}}}}
\newcommand{\Hwft}{\mathcal{H}_{F,\tau}}
\newcommand{\Hwftm}{\mathcal{H}_{F,\tau}^{(m)}}

\newcommand{\car}{{\mathfrak{h}}}    
\newcommand{\bor}{{\mathfrak{b}}}    
\newcommand{\nil}{{\mathfrak{n}}}    
\newcommand{\vp}{{\varphi}}
\newcommand{\bh}{\widehat{\mathfrak{b}}}  
\newcommand{\bb}{\overline{\mathfrak{b}}}  
\newcommand{\Vh}{\widehat{\mathcal V}}
\newcommand{\KZ}{Kniz\-hnik-Zamo\-lod\-chi\-kov}
\newcommand{\TUY}{Tsuchia, Ueno  and Yamada}
\newcommand{\KN} {Kri\-che\-ver-Novi\-kov}
\newcommand{\pN}{\ensuremath{(P_1,P_2,\ldots,P_N)}}
\newcommand{\xN}{\ensuremath{(\xi_1,\xi_2,\ldots,\xi_N)}}
\newcommand{\lN}{\ensuremath{(\lambda_1,\lambda_2,\ldots,\lambda_N)}}
\newcommand{\iN}{\ensuremath{1,\ldots, N}}
\newcommand{\iNf}{\ensuremath{1,\ldots, N,\infty}}

\newcommand{\tb}{\tilde \beta}
\newcommand{\tk}{\tilde \varkappa}
\newcommand{\ka}{\kappa}
\renewcommand{\k}{\varkappa}
\newcommand{\ce}{{c}}

\newcommand{\Pif} {P_{\infty}}
\newcommand{\Pinf} {P_{\infty}}
\newcommand{\PN}{\ensuremath{\{P_1,P_2,\ldots,P_N\}}}
\newcommand{\PNi}{\ensuremath{\{P_1,P_2,\ldots,P_N,P_\infty\}}}
\newcommand{\Fln}[1][n]{F_{#1}^\lambda}
\newcommand{\tang}{\mathrm{T}}
\newcommand{\Kl}[1][\lambda]{\can^{#1}}
\newcommand{\A}{\mathcal{A}}
\newcommand{\U}{\mathcal{U}}
\newcommand{\V}{\mathcal{V}}
\newcommand{\W}{\mathcal{W}}
\renewcommand{\O}{\mathcal{O}}
\newcommand{\Ae}{\widehat{\mathcal{A}}}
\newcommand{\Ah}{\widehat{\mathcal{A}}}
\newcommand{\La}{\mathcal{L}}
\newcommand{\Le}{\widehat{\mathcal{L}}}
\newcommand{\Lh}{\widehat{\mathcal{L}}}
\newcommand{\eh}{\widehat{e}}
\newcommand{\Da}{\mathcal{D}}
\newcommand{\kndual}[2]{\langle #1,#2\rangle}
\newcommand{\cins}{\frac 1{2\pi\mathrm{i}}\int_{C_S}}
\newcommand{\cinsl}{\frac 1{24\pi\mathrm{i}}\int_{C_S}}
\newcommand{\cinc}[1]{\frac 1{2\pi\mathrm{i}}\int_{#1}}
\newcommand{\cintl}[1]{\frac 1{24\pi\mathrm{i}}\int_{#1 }}
\newcommand{\w}{\omega}
\newcommand{\ord}{\operatorname{ord}}
\newcommand{\res}{\operatorname{res}}
\newcommand{\nord}[1]{:\mkern-5mu{#1}\mkern-5mu:}
\newcommand{\codim}{\operatorname{codim}}
\newcommand{\ad}{\operatorname{ad}}
\newcommand{\Ad}{\operatorname{Ad}}
\newcommand{\supp}{\operatorname{support}}

\newcommand{\Fn}[1][\lambda]{\mathcal{F}^{#1}}
\newcommand{\Fl}[1][\lambda]{\mathcal{F}^{#1}}
\renewcommand{\Re}{\mathrm{Re}}

\newcommand{\ha}{H^\alpha}

\define\ldot{\hskip 1pt.\hskip 1pt}
\define\ifft{\qquad\text{if and only if}\qquad}
\define\a{\alpha}
\redefine\d{\delta}
\define\w{\omega}
\define\ep{\epsilon}
\redefine\b{\beta} \redefine\t{\tau} \redefine\i{{\,\mathrm{i}}\,}
\define\ga{\gamma}
\define\cint #1{\frac 1{2\pi\i}\int_{C_{#1}}}
\define\cintta{\frac 1{2\pi\i}\int_{C_{\tau}}}
\define\cintt{\frac 1{2\pi\i}\oint_{C}}
\define\cinttp{\frac 1{2\pi\i}\int_{C_{\tau'}}}
\define\cinto{\frac 1{2\pi\i}\int_{C_{0}}}
\define\cinttt{\frac 1{24\pi\i}\int_C}
\define\cintd{\frac 1{(2\pi \i)^2}\iint\limits_{C_{\tau}\,C_{\tau'}}}
\define\dintd{\frac 1{(2\pi \i)^2}\iint\limits_{C\,C'}}
\define\cintdr{\frac 1{(2\pi \i)^3}\int_{C_{\tau}}\int_{C_{\tau'}}
\int_{C_{\tau''}}}
\define\im{\operatorname{Im}}
\define\re{\operatorname{Re}}
\define\res{\operatorname{res}}
\redefine\deg{\operatornamewithlimits{deg}}
\define\ord{\operatorname{ord}}
\define\rank{\operatorname{rank}}
\define\fpz{\frac {d }{dz}}
\define\dzl{\,{dz}^\l}
\define\pfz#1{\frac {d#1}{dz}}

\define\K{\Cal K}
\define\U{\Cal U}
\redefine\O{\Cal O}
\define\He{\text{\rm H}^1}
\redefine\H{{\mathrm{H}}}
\define\Ho{\text{\rm H}^0}
\define\A{\Cal A}
\define\Do{\Cal D^{1}}
\define\Dh{\widehat{\mathcal{D}}^{1}}
\redefine\L{\Cal L}
\newcommand{\ND}{\ensuremath{\mathcal{N}^D}}
\redefine\D{\Cal D^{1}}
\define\KN {Kri\-che\-ver-Novi\-kov}
\define\Pif {{P_{\infty}}}
\define\Uif {{U_{\infty}}}
\define\Uifs {{U_{\infty}^*}}
\define\KM {Kac-Moody}
\define\Fln{\Cal F^\lambda_n}
\define\gb{\overline{\mathfrak{ g}}}
\define\G{\overline{\mathfrak{ g}}}
\define\Gb{\overline{\mathfrak{ g}}}
\redefine\g{\mathfrak{ g}}
\define\Gh{\widehat{\mathfrak{ g}}}
\define\gh{\widehat{\mathfrak{ g}}}
\define\Ah{\widehat{\Cal A}}
\define\Lh{\widehat{\Cal L}}
\define\Ugh{\Cal U(\Gh)}
\define\Xh{\hat X}
\define\Tld{...}
\define\iN{i=1,\ldots,N}
\define\iNi{i=1,\ldots,N,\infty}
\define\pN{p=1,\ldots,N}
\define\pNi{p=1,\ldots,N,\infty}
\define\de{\delta}

\define\kndual#1#2{\langle #1,#2\rangle}
\define \nord #1{:\mkern-5mu{#1}\mkern-5mu:}
\newcommand{\MgN}{\mathcal{M}_{g,N}} 
\newcommand{\MgNeki}{\mathcal{M}_{g,N+1}^{(k,\infty)}} 
\newcommand{\MgNeei}{\mathcal{M}_{g,N+1}^{(1,\infty)}} 
\newcommand{\MgNekp}{\mathcal{M}_{g,N+1}^{(k,p)}} 
\newcommand{\MgNkp}{\mathcal{M}_{g,N}^{(k,p)}} 
\newcommand{\MgNk}{\mathcal{M}_{g,N}^{(k)}} 
\newcommand{\MgNekpp}{\mathcal{M}_{g,N+1}^{(k,p')}} 
\newcommand{\MgNekkpp}{\mathcal{M}_{g,N+1}^{(k',p')}} 
\newcommand{\MgNezp}{\mathcal{M}_{g,N+1}^{(0,p)}} 
\newcommand{\MgNeep}{\mathcal{M}_{g,N+1}^{(1,p)}} 
\newcommand{\MgNeee}{\mathcal{M}_{g,N+1}^{(1,1)}} 
\newcommand{\MgNeez}{\mathcal{M}_{g,N+1}^{(1,0)}} 
\newcommand{\MgNezz}{\mathcal{M}_{g,N+1}^{(0,0)}} 
\newcommand{\MgNi}{\mathcal{M}_{g,N}^{\infty}} 
\newcommand{\MgNe}{\mathcal{M}_{g,N+1}} 
\newcommand{\MgNep}{\mathcal{M}_{g,N+1}^{(1)}} 
\newcommand{\MgNp}{\mathcal{M}_{g,N}^{(1)}} 
\newcommand{\Mgep}{\mathcal{M}_{g,1}^{(p)}} 
\newcommand{\MegN}{\mathcal{M}_{g,N+1}^{(1)}} 

\define \sinf{{\widehat{\sigma}}_\infty}
\define\Wt{\widetilde{W}}
\define\St{\widetilde{S}}
\newcommand{\SigmaT}{\widetilde{\Sigma}}
\newcommand{\hT}{\widetilde{\frak h}}
\define\Wn{W^{(1)}}
\define\Wtn{\widetilde{W}^{(1)}}
\define\btn{\tilde b^{(1)}}
\define\bt{\tilde b}
\define\bn{b^{(1)}}
\define \ainf{{\frak a}_\infty} 

%
\define\eps{\varepsilon}    
\newcommand{\e}{\varepsilon}
\define\doint{({\frac 1{2\pi\i}})^2\oint\limits _{C_0}
       \oint\limits _{C_0}}                            
\define\noint{ {\frac 1{2\pi\i}} \oint}   
\define \fh{{\frak h}}     
\define \fg{{\frak g}}     
\define \GKN{{\Cal G}}   
\define \gaff{{\hat\frak g}}   
\define\V{\Cal V}
\define \ms{{\Cal M}_{g,N}} 
\define \mse{{\Cal M}_{g,N+1}} 
\define \tOmega{\Tilde\Omega}
\define \tw{\Tilde\omega}
\define \hw{\hat\omega}
\define \s{\sigma}
\define \car{{\frak h}}    
\define \bor{{\frak b}}    
\define \nil{{\frak n}}    
\define \vp{{\varphi}}
\define\bh{\widehat{\frak b}}  
\define\bb{\overline{\frak b}}  
\define\KZ{Knizhnik-Zamolodchikov}
\define\ai{{\alpha(i)}}
\define\ak{{\alpha(k)}}
\define\aj{{\alpha(j)}}
\newcommand{\calF}{{\mathcal F}}
\newcommand{\ferm}{{\mathcal F}^{\infty /2}}
\newcommand{\Aut}{\operatorname{Aut}}
\newcommand{\End}{\operatorname{End}}
\newcommand{\laxgl}{\overline{\mathfrak{gl}}}
\newcommand{\laxsl}{\overline{\mathfrak{sl}}}
\newcommand{\laxso}{\overline{\mathfrak{so}}}
\newcommand{\laxsp}{\overline{\mathfrak{sp}}}
\newcommand{\laxs}{\overline{\mathfrak{s}}}
\newcommand{\laxg}{\overline{\frak g}}
\newcommand{\bgl}{\laxgl(n)}
\newcommand{\tX}{\widetilde{X}}
\newcommand{\tY}{\widetilde{Y}}
\newcommand{\tZ}{\widetilde{Z}}


\title[Matrix divisors]{Matrix divisors on Riemann surfaces and Lax operator algebras}
\author[O.K. Sheinman]{O.K. Sheinman}
\thanks{Partial  support by the
Internal Research Project  GEOMQ11,  University of Luxembourg,
and
by the OPEN scheme of the Fonds National de la Recherche
(FNR), Luxembourg,  project QUANTMOD O13/570706
is gratefully acknowledged.}

\maketitle


\tableofcontents
\section{Introduction}
Matrix divisors are introduced in the work by A.Weil \cite{Weil} which is considered as a starting point of the theory of holomorphic vector bundles on Riemann surfaces. The classification of the holomorphic vector bundles on Riemann surfaces by A.N.Tyurin \cite{Tyur64,Tyur65,Tyur66} based on matrix divisors, the well-known Narasimhan--Seshadri description of stable vector bundles \cite{N-S}, and subsequent description of the moduli space of vector bundles with the parabolic structure \cite{Seshadri,MehtaS} originate in \cite{Weil}. In the theory of holomorphic vector bundles the matrix divisors play the role similar to the role of usual divisors in the theory of line bundles.

The matrix divisor approach to classification of holomorphic vector bundles provides invariants not only of stable bundles but also of families of smaller dimensions. Moreover, it provides explicit coordinates, invented in \cite{Tyur65}, in an open subset of the moduli space of stable vector bundles. In \cite{rKNU}, these coordinates were given the name of \emph{Tyurin parameters} and applied in integration of soliton equations.

To be more specific, assume that a holomorphic rank $n$ vector bundle has the $n$-dimensional space of holomorphic sections. Then any base of the space of the holomorphic sections is called \emph{framing}, and the bundle with a given framing is called a \emph{framed bundle}. The classification of the framed holomorphic vector bundles is one of the main results of \cite{Tyur64,Tyur65,Tyur66}. In particular, it follows from \cite{Tyur65,Tyur66} that the moduli space of stable framed rank $n$ holomorphic vector bundles of degree $0$ is a quasiprojective variety of dimension $(n^2-1)(g-1)$ where $g\ge 2$ is the genus of the Riemann surface, and if in the same set-up we consider the bundles of degree $ng$ then the dimension of the corresponding quasiprojective variety is equal to $n^2(g-1)+1$. It has been also shown by Tyurin that the bundles which do not possess any natural framing depend on a smaller number of parameters.

In the present paper, we address the problem of classifying the matrix divisors. It is a straight forward generalization of the problem of classifying the framed vector bundles. Indeed, let $\psi_1^U,\ldots, \psi_n^U$ be the elements of a framing represented in local coordinates (i.e. the local meromorphic vector-functions defined at the local coordinate set $U$). Then the collection of matrices $\Psi^U$ formed by them at every $U$ form a matrix divisor.

We would like to gain attention to one more relationship between matrix divisors and the theory of integrable systems, namely to the relationship with Lax operator algebras. Those came to existence due to the theory by Krichever \cite{Klax} of integrable systems with the spectral parameter on a Riemann surface. Originally, this theory has been motivated in part by the Tyurin parametrization of framed vector bundles. In \cite{Sh_TrGr,Sh_UMN_2015} it has been developed in the different,  and more general set-up related to $\Z$-gradings of the semisimple Lie algebras. The main purpose of the present work is to develop the corresponding set-up in the theory of matrix divisors. The result we obtain on this way can be briefly formulated as follows.
\begin{Thm}
The moduli space $\M$ of matrix divisors with certain discrete invariants and fixed support is a homogeneous space. For its tangent space at the unit we have
\begin{equation}\label{E:quotq}
 T_e\M\cong \M^\L/\L
\end{equation}
where $\L$ is the Lax operator algebra essentially defined by the same invariants, $\M^\L$  is the corresponding space of $M$-operators.
\end{Thm}
\noindent
This result goes back to \cite{Klax}. We refer to \refS{globa}, in particular to \refT{quot}, for the details, notation, and a more precise statement.

We were not able to find out any reference for the matrix divisors of $G$-bundles for $G$ a complex semisimple group. It is one of the purposes of the present work, closely related to the main purpose, to propose a treatment of such matrix divisors. To do that, we use the Chevalley groups over the field (resp. ring) of Laurent (resp. Tailor) series. It is a very adequate set-up for matrix divisors by our opinion, because a Chevalley group is defined by a (complex, semisimple) Lie algebra and its faithful representation given by a highest weight. Such data contain information both on the group structure and on the fibre of the bundle. Moreover, the Cartan decomposition of Chevalley groups in its general form  provides a convenient description of the canonical form of a matrix divisor (Theorems \ref{T:Cheval},\ref{T:Cartan} below). The Cartan decomposition, in particular, states that for an arbitrary Chevalley group $G$ over the field of Laurent series the following holds: $G=KA^+K$ where $K$ is the same group considered over the ring of Tailor series, and $A^+$ is a chamber in the maximal torus. In \cite{Tyur65} the same role is played by Lemma 1.2.1. However, the last claims a stronger statement, namely it specifies the form of the $K$-component of the decomposition in the following quite beautiful way:
let $k$ be the $K$-component at a certain point of the divisor support, $diag(z^{d_1},\ldots,z^{d_n})\in A^+$ be the toric component, $d_1\le \ldots\le d_n$, $E_{ij}$ are the matrix units, then
\[
          k=E+\sum_{i<j}a_{ij}(z)E_{ij},\quad a_{ij}(z)\in \C(z)/z^{d_j-d_i}\C(z)
\]
where $\C(z)$ is the ring of Tailor series. Since we are not able to follow all the arguments by A.N.Tyurin in course of deriving that expression, we reinterpret it,  generalize it to the case of an arbitrary reduced root system $R$, and thus obtain the following description of the tangent space to the moduli space of matrix divisors (see \refT{mspace} below for the more precise statement).
\begin{Thm}
The tangent space to $\M$ at the unit consists of elements of the form
\[
    \bigoplus_{\ga\in\Gamma}\sum_{\a\in R^+} a_\a^\ga(z)x_\a,\quad a_\a^\ga(z)\in \C(z)/z^{\a(h_\ga)}\C(z)
\]
\end{Thm}
\noindent
where $\Gamma$ is the divisor support, $h_\ga$ comes from the maximal torus component of the Cartan decomposition at $\ga$.
Finally it turns out to be an important argument for establishing the above relationship (given by \refE{quotq}) between matrix divisors and Lax operator algebras.

In the present paper we assume $G$ to be semi-simple which corresponds to the case of topologically trivial holomorphic vector bundles. To include the topologically non-trivial bundles we would need to consider the conformal extensions of semi-simple groups \cite{LOSZ} instead. Let $G$ be a complex semi-simple group with the finite center $\mathcal Z$ equal to a direct sum of $r$ cyclic components. By \emph{conformal extension} of $G$ we call $G_c=G\times_{\mathcal Z}(\C^*)^r$. For example, $GL(n,\C)$ is a conformal extension of $SL(n,\C)$. We do not focus on this easy modification here.

The plan of the present paper is as follows. In \refS{mdiv} we give the preliminaries on matrix divisors, our treatment of this notion related to Chevalley groups, and a description of the space of sections of a matrix divisor in terms of certain flag configurations. In \refS{canon} we define the moduli space of matrix divisors as a certain coset, and prove \refT{mspace} giving a description of its tangent space at the unit in terms of the root system of the group, and the weight lattice of the underlying module. In \refS{globa} we give preliminaries on Lax operator algebras (see \cite{Sh_UMN_2015} for the details) and then complete the interpretation of the moduli space from the point of view of integrable systems identifying the tangent space at the unit with the coset of the space of $M$-operators by the space of $L$-operators in spirit of \cite{Klax}, relying on the results of \cite{Sh_UMN_2015}.

I would like to express my gratitude to I.M.Krichever whom I am indebted with my interest to the subject, and who is a pioneer of many ideas relating integrable systems and holomorphic vector bundles on Riemann surfaces. I am grateful to M.Schlichenmaier with whom we started to discuss the subject many years ago, and to E.M.Chirka for his help in complex analysis. I would like to acknowledge a special role of discussions with E.B.Vinberg on Lax operator algebras and algebraic groups.
\section{Matrix divisors and flag configurations}\label{S:mdiv}
Let $G$ denote a Chevalley group given by a semi-simple complex Lie algebra  $\g$, a faithful $\g$-module $V$ with dominant highest weight, and a field $k$. We recall that $G$ is the group of automorphisms of the $k$-space $V^k=V\otimes_\Q k$ generated by the 1-parameter subgroups of automorphisms of the form $\exp tX_\a = \sum\limits_{n=0}^\infty t^nX_\a^n/n!$ (the sums being actually finite on $V^k$) where $X_\a$ is the root vector of the root $\a$.

Let $\Sigma$ be a Riemann surface.

The following system of definitions reproduces the corresponding definitions in \cite{Tyur65}.
\begin{definition}
Assume each point of $\Sigma$ to be assigned with a germ of meromorphic $G$-valued functions holomorphic except at a finite set $\Gamma\subset\Sigma$. Such a correspondence is called \emph{distribution} with the support~$\Gamma$.
\end{definition}
\begin{definition}
Two distributions $A_x$ and $B_x$, $x\in\Sigma$ are \emph{equivalent} if there is a third distribution $C_x$ holomorphic for every $x\in\Sigma$ and such that $C_xA_x=B_x$. Any class of equivalent distributions is called \emph{matrix divisor}.
\end{definition}
That $C_x$ is holomorphic and $G$-valued implies in particular that it is holomorphically invertible.

We delay the discussion of equivalent matrix divisors until \refS{canon} (\refD{equid} and below).
\begin{definition}\label{D:sect}
Given a matrix divisor $\Psi$, by its  \emph{local section} (or just \emph{section}) we mean a meromorphic $V$-valued function $f$ on an open subset $U\subset\Sigma$ such that $f$ is holomorphic on $U\backslash\Gamma$ and $\Psi_\ga f$ is holomorphic in the neighborhood of any $\ga\in \Gamma\cap U$.
\end{definition}
We denote the sheaf of sections by $\Gamma_V(\Psi)$. It has a simple description in terms of flag configurations related to the divisor.

Given a matrix divisor $\Psi$ we assign a flag in $V$ to every point in its support. Thus the divisor turns out to be assigned with the system of flags which we call a \emph{flag configuration} (this term assumes $\Gamma$ to be fixed).

Let $f$ be a meromorphic $V$-valued function on $U\subset\Sigma$. In order $f$ be a section it is required that
\begin{equation}\label{E:sect}
         s=\Psi f
\end{equation}
is holomorphic at every $\ga\in\Gamma\cap U$ where $\Gamma=\operatorname{ support}\Psi$. Assume $\Psi$ to have an expansion of the form
\[
    \Psi=\sum_{i=-m}^\infty \Psi_iz^i
\]
at $\ga$, and $f$ to have an expansion of the form
\[
        f=\sum_{i=-k}^\infty f_iz^i
\]
there. We take $-k=\ord_\ga\Psi_\ga^{-1}$ which follows from \refE{sect}. We then have the following system of $m+k$ linear equations
\begin{equation}\label{E:syst}
\begin{aligned}
  & \Psi_{-m}f_{-k}=0,\quad  \Psi_{-m}f_{-k+1}+\Psi_{-m+1}f_{-k}=0,\ \ldots,
  \\
  & \Psi_{-m} f_{m-1}+\Psi_{-m+1} f_{m-2}+\ldots+ \Psi_{k-1} f_{-k}=0\ .
\end{aligned}
\end{equation}
expressing the fact that the terms containing $z^{-m-k},\ldots,z^{-1}$ of $\Psi f$ vanish, i.e. $s$ in \refE{sect} is holomorphic. This system of equations is homogeneous, hence all the components of its solutions constitute linear spaces. Let $F_i$ be the subspace in $V$ constituted by $f_i$'s for all solutions $f_{-k},f_{-k+1},\ldots,f_{m-1}$ to \refE{syst}.

The following lemma claims that the subspaces $F_i$, $i=-k,\ldots,m-1$ constitute a flag; we were not able to find any reference for it. Similar arguments are used for the flag interpretation of opers \cite{Feig}.
\begin{lemma}\label{L:flag}
     $F_{-k}\subseteq F_{-k+1}\subseteq\ldots\subseteq F_{m-1}\subseteq V$.
\end{lemma}
\begin{proof} Solutions to \refE{syst} can be found out by an inductive procedure. The first equation is independent and homogeneous. The space of its solutions is exactly what we have denoted by $F_{-k}$. Solutions to the second equation can be found out by plugging an arbitrary $f_{-k}\in F_{-k}$, and resolving the obtained system of equations. In particular, we can plug $f_{-k}=0$. Then the second equation coincides with the first one, hence $F_{-k}\subseteq F_{-k+1}$.

The $i$th equation of \refE{syst} (at $z^{-m-k+i}$) is as follows:
\[
    \Psi_{-m} f_{-k+i}+\Psi_{-m+1} f_{-k+i-1}+\ldots+ \Psi_{-m+i} f_{-k}=0.
\]
The next equation has the form
\[
    \Psi_{-m} f_{-k+i+1}+\Psi_{-m+1} f_{-k+i}+\ldots+ \Psi_{-m+i} f_{-k+1}+\Psi_{-m+i+1} f_{-k}=0.
\]
If $f_{-k}=0$, the $(i+1)$th equation degenerates to the $i$th one. Hence every solution of the $i$th equation is a particular solution of the $(i+1)$th equation, i.e. $F_{-m-k+i}\subseteq F_{-m-k+i+1}$.
\end{proof}
The description of the sheaf $\Gamma_V(\Psi)$ is now as follows.
\begin{lemma}\label{L:sect}
$\Gamma_V(\Psi)$ is the sheaf of local meromorphic $V$-valued functions on $\Sigma$ satisfying the following requirement for every $\ga\in\Gamma$. Let $f$ be such a function, and $f(z)=\sum f_i^\ga z^i$ be its Laurent expansion at a $\ga\in\Gamma$. Then it is required that $f_i^\ga\in F_i^\ga$ where $F_\ga:\,\{ 0\}\subseteq\ldots\subseteq F_i^\ga\subseteq\ldots\subseteq V$ is the flag corresponding to $\ga$.
\end{lemma}
The proof immediately follows from \refD{sect} and the definition of $F$. In the \refL{sect} we consider the flag to be semi-infinite to the right, where $F_i$ stabilize to $V$ since a certain moment.
\begin{definition}
  Given a matrix divisor $\Psi$ we call the Lie algebra of meromorphic $\g$-valued functions on $\Sigma$ leaving invariant $\Gamma_V(\Psi)$ by \emph{endomorphism algebra} of $\Psi$, and denote it by ${\End}(\Psi)$.
\end{definition}
Next we give a description of the Lie algebra $\End(\Psi)$ in terms of the flag configuration related to $\Psi$.

Let $\g=Lie(G)$. Given a flag $F$ consider the following filtration of~$\g$. Remind that $V$ is a $\g$-module. For every $i$ consider a subspace $\tilde\g_i\subseteq\g$ such that $\tilde\g_iF_j\subseteq F_{j+i}$ for every $j$.  Then $\tilde\g_i\subseteq\tilde\g_{i+1}$ because $\tilde\g_iF_j\subseteq F_{j+i}\subseteq F_{j+i+1}$.
\begin{lemma}\label{L:Endo}
$\End(\Psi)$ is the subspace of the space of all $\g$-valued merom\-orphic functions on $\Sigma$ satisfying the following requirement for every $\ga\in\Gamma$. Let $L$ be such a function, and $L(z)=\sum L_iz^i$ be its Laurent expansion at a $\ga\in\Gamma$. Then $L_i\in \tilde\g_i,\ \forall i$.
\end{lemma}

It is instructive to keep in mind the homological interpretation of matrix divisors \cite{Tyur2d,TyurUMN}. In this approach a matrix divisor is defined as a 0-cochain with coefficients in the sheaf $G({\mathcal R})$ of rational $G$-valued functions whose boundary is a 1-cocycle with coefficients in the sheaf $G({\mathcal O})$ of regular $G$-valued functions. Thus an open covering $\{U_i\}$ of $\Sigma$ is assigned with the system of local rational $G$-valued functions $f_i$ such that its boundary $f_{ij}=f_if_j^{-1}$ is regular and regularly invertible on $U_i\cap U_j$, and $f_{ij}f_{jk}f_{ki}=1$ on $U_i\cap U_j\cap U_k$ for every triple $(i,j,k)$.  It is clear that the cocycle $f_{ij}$ (the system of gluing functions in other terminology) is invariant with respect to the right action of any global $G$-valued function.
\section{Canonical form of a matrix divisor. The moduli space}\label{S:canon}

Let $\frak o$ be a principal ideal ring (commutative, with the unit), $\frak k$ be its quotient field, $G$ be a rank $l$ Chevalley group over $\frak k$, $K=G_{\frak o}$ be the same group over $\frak o$, $H$ be a maximal torus in $G$, i.e. the subgroup generated by the 1-parameter subgroups $h'_i(t)$ where $h'_i(t)$ acts on every vector of a weight $\mu$ in $V$ as a multiplication by $t^{\mu(h_i)}$, $h_i\in\h$ ($i=1,\ldots,l$) form a base of the lattice $L_V^*$ dual to $L_V$, the last being generated by the weights of the module $V$ \cite[Lemma 35, p.58]{Stein}.

For an obvious reason $h_i'(t)$ is denoted also by $t^{h_i}$, and $H$ consists of the elements of the form $t_1^{h_1}\cdot\ldots\cdot t_l^{h_l}$ where
$t_i\in\frak k\setminus \{ 0\}$, $i=1,\ldots,l$. The following example shows how the module $V$ affects the torus of the corresponding Chevalley group.
\begin{example}
Consider $\g=\sln (2,\C)$.  Up to the end of the example let $h$ denote the canonical generator of the Cartan subalgebra. The standard $\g$-module has the weights $\pm 1$, and the adjoint module the weights $\{ \pm 2,0\}$. Hence the dual lattice to the weight lattice $L_V$ is generated by $h$ in the first case, and by $(1/2)h$ in the second case. The corresponding tori consist of elements of the form $t^{\rho(h)n}$ in the first case, and  $t^{\rho(h)n/2}$ in the second case where $n\in\Z$ and $\rho(h)$ denotes the corresponding representation operator (i.e. $\rho(h)=diag(1,-1)$ in the first case, and $\rho(h)=diag(2,0,-2)$ in the second case).
\end{example}

By definition of a Chevalley group it is assumed that the representation of $\g$ in $V$ is a faithful representation. For the reason that the root lattice is always a sublattice of $L_V$ we have $\a(h_i)\in\Z$ for every $\a\in R$, $i=1,\ldots,l$. Let $A^+$ denote the chamber in $H$ given by the tuples $\{ t_1,\ldots ,t_l\}$ satisfying the condition $\prod_{i=1}^l t_i^{\a(h_i)}\in\frak o$ for every positive root $\a$.
\begin{theorem}[\cite{Stein}, Theorem 21]\label{T:Cheval}
\begin{itemize}
\item[]
\item[(a)] $G=KA^+K$ (the Cartan decomposition);
\item[(b)] The $A^+$-component in (a) is defined uniquely modulo $H\cap K$.
\end{itemize}
\end{theorem}
We need a particular case of \refT{Cheval} when $\frak o=\C(z)$ is the ring of Tailor expansions in $z$, $\frak k=\C(z^{-1},z)$ is the field of the Laurent expansions. The elements $t_i$ up to units of the ring $\C(z)$ can be taken as $t_i=z^{d_i}$ where $d_i\in\Z$ ($i=1,\ldots,l$). Let $h=\sum_{i=1}^ld_ih_i$. Then $t_1^{h_1}\cdot\ldots\cdot t_l^{h_l}=z^h$ (once again, $z^h$ is defined on a module $V$ such that $h\in L_V^*$, and operates as multiplication by $z^{\mu(h)}$ on the subspace of weight $\mu$ in $V$). The condition  $\prod_{i=1}^l t_i^{\a(h_i)}\in\frak o$ means that $z^{\a(h)}$ is holomorphic, i.e. $\a(h)\ge 0$, for every positive $\a$.
We conclude that $H=\{ z^h|h\in L_V^* \}$ and $A^+=\{ z^h| \a(h)\ge 0,\ \forall \a\in R_+ \}$. In this case we obtain the following Cartan decomposition for current groups out of \refT{Cheval}. In certain particular cases it is also known as factorization theorem in the theory of holomorphic vector bundles.
\begin{theorem}\label{T:Cartan}
Let $G=G(\C(z^{-1},z))$ be the Chevalley group given by a simple Lie algebra $\g$ over $\C$ and its module $V$, $K=G(\C(z))$ be the same group over the ring of Taylor series, $H$ and $A^+$ be as introduced above. Then
\[  G=KA^+K,
\]
and the $A^+$-component in the decomposition is determined uniquely up to $H_\C$.
\end{theorem}

By virtue of \refT{Cartan} the support of a divisor can be characterized as the set of those points in $\Sigma$ for which $h\ne 0$.

Up to equivalence given by left multiplication by a distribution taking values in $K$ we can assume that $\Psi=z^hk(z)$, $k(z)\in K$ where $h$ and $k$ depend on the point of $\operatorname{supp}\Psi$. We call it the \emph{reduced form} of the divisor.

Given a $\g$-module $V$ of highest weight $\chi$, and an $h\in (L_V^*)_+$ we introduce the following flag $F$ in $V$ (below $m=\chi(h)$, $m\in\Z_+$). First we define the grading of the module $V$:
\begin{equation}\label{E:vgr}
    V=\bigoplus_{i=-m}^m V_i\quad \text{where}\quad V_i=\{ v\in V \ |\ hv=-iv \}.
\end{equation}
Obviously $V_i=\bigoplus\limits_{\mu(h)=-i} V_\mu$, $V_\mu$ being the weight subspace of $V$ of weight~$\mu$.

Next we define the flag $F\, :\, \{ 0\}\subseteq F_{-m}\subseteq\ldots\subseteq F_m=V$ by setting
\begin{equation}\label{E:hflag}
   F_j=\bigoplus_{s=-m}^j V_s.
\end{equation}
In particular, $F_{-m}$ is generated by the highest weight vector.
Let $Q=\Z(R)$ be the root lattice of $\g$.
\begin{lemma}
Let $h\in  L_V^*$. Then $F$ is nothing but the flag corresponding to the divisor $\Psi=z^h$ by virtue of \refL{flag}.
\end{lemma}
\begin{proof}
We resolve the equation \refE{sect} for $\Psi=z^h$:\ $f=z^{-h}s$ where $s(z)$ is holomorphic in the neighborhood of $z=0$.
Take $s(z)=\sum\limits_{j=0}^\infty s_jz^j$ where $s_j\in V$ for every $j\ge 0$.
Let $s_j=\sum\limits_{i=-m}^m s_j^{(i)}$ be an expansion of $s_j$ according to the grading of $V$, i.e. $s_j^{(i)}\in V_i$. Then
\[
  z^{-h}s_j=\sum\limits_{i=-m}^m s_j^{(i)}z^i,
\]
hence
\[
    f(z)= \sum_{j=0}^\infty\left( \sum_{i=-m}^m s_j^{(i)}z^i
                  \right)z^j
                = \sum_{p=-m}^\infty\left( \sum_{i=-m}^p s_{p-i}^{(i)}
                  \right)z^p.
\]
Since $i\le p$ in the internal sum, the last belongs to the subspace $F_p$.
\end{proof}
\begin{remark}
The flags of the form \refE{hflag} already occured in \cite{Falt} in the context of \emph{infinitesimal parabolic structures}. In contrast to any parabolic structure our flag distributions appear as an intrinsic structure for a matrix divisor and the corresponding holomorphic vector bundle.
\end{remark}
Let $D$ be a nonnegative divisor, $\Pi=\supp D$,  $\Pi\cap\Gamma=\varnothing$, and $\Gamma_{\rm gl}^D(\Psi)=\{ f\in \Gamma_V(\Psi)\ |\ (f)+D+m\Gamma \ge 0,\ f\ \text{is global} \}$.
\begin{corollary}\label{C:nontr}
$\dim \Gamma_{\rm gl}^D(\Psi)=\dim V (\deg D-g+1)$, in particular $\Gamma_{\rm gl}^D(\Psi)$ is trivial unless $\deg D\ge g$.
\end{corollary}
\begin{proof}
Let $\F^{D+m\Gamma}$ denote the space of global sections $f$ satisfying the condition $(f)+D+m\Gamma \ge 0$, and $l_{D+m\Gamma}=\dim\F^{D+m\Gamma}$. By the Riemann--Roch theorem $l_{D+m\Gamma}=(\dim V)(\deg D+m|\Gamma |-g+1)$. However, the space of sections has a codimension in $\F^{D+m\Gamma}$ coming from the conditions $f_s\in\F_s^\ga$ where $F_s^{\ga}$, $s=-m,\ldots,m$ are the flag subspaces at~$\ga$. The contribution of every $\ga\in\Gamma$ to the codimension is equal to $\sum_{s=-m}^m \codim_VF_s^{\ga}$. By symmetry of the grading \refE{vgr} the last is equal to $m\dim V$. The total codimension is equal to  $(m\dim V)|\Gamma|$. Hence the rest of dimension is equal to $(\dim V)(\deg D-g+1)$, and it should be $\deg D-g+1>0$ for non-triviality of the space of sections.
\end{proof}
The highest weight $\chi$, and the tuple $h=\{ h_\ga\in L_V^*\ |\ \ga\in\Gamma\}$ are discrete invariants of a divisor. The matrix divisors also have moduli coming from the $K$-components of their canonical forms at the points in $\Gamma$, and from elements $\ga\in\Gamma$ themselves. Below we introduce two types of  equivalence of matrix divisors, and the corresponding moduli spaces.
\begin{definition}\label{D:equid}
Two matrix divisors are \emph{equivalent} if they have the same sheaf of sections (up to the common left shift by a constant (in $z$ and $\ga$) element of $G$).
\end{definition}
\begin{remark}
This equivalence is different from that given in \cite{Tyur65,Tyur66} for the purpose of classificaton of the holomorphic vector bundles. Following the last, two matrix divisors are equivalent if one of them can be taken to another by right multiplication by a (global) meromorphic $G$-valued function. In particular, the divisors equivalent in this sense may have different support.
\end{remark}
\begin{lemma}
Two matrix divisors are equivalent in sense of \refD{equid} if, and only if, they have the same flag configuration (up to the common left shift by a constant (in $z$ and $\ga$) element of $G$), in particular, the same $\Gamma$.
\end{lemma}
\begin{proof}
Immediately follows from \refL{sect}.
\end{proof}
We define a \emph{flag set} in the same way as a flag configuration except that we relax the requirement that $\Gamma$ is fixed. Then our second equivalence relation is as follows.
\begin{definition}\label{D:equid2}
Two matrix divisors are equivalent if they have the same flag set (up to the common left shift by a constant (in $z$ and $\ga$) element of $G$).
\end{definition}
Denote the moduli space of matrix divisors with given invariants $\chi$, $h=\{ h_\ga\in Q^*\cap L_V\ |\ \ga\in\Gamma \}$ by $\M^\chi_h$ (it corresponds to the equivalence given by \refD{equid2}), and with additionally fixed $\Gamma$ by $\M^{\chi,\Gamma}_h$ (it corresponds to the equivalence given by \refD{equid}). Denote the part of the moduli space over a point $\ga\in\Gamma$ by $\M^\chi_{h_\ga}$.

It is our next step to represent $\M^{\chi,\Gamma}_h$ as a homogeneous space and describe the stationary group of a point of this space.

Given a matrix divisor, consider the corresponding $h=\{ h_\ga\in L_V^* |\ \ga\in\Gamma \}$. Observe that for a faithful $\g$-module $V$ one has $L_V^*\subseteq Q^*$ \cite[Lemma 27]{Stein}. Hence $\a(h_\ga)\in\Z$ for every $\a\in R$, $\ga\in\Gamma$. For every $\ga$ the $h_\ga$ defines a grading $\g=\g_{-k}^\ga\oplus\ldots\oplus\g_k^\ga$, and the corresponding increasing filtration $\tilde\g_{-k}^\ga\subset\ldots\subset\tilde\g_k^\ga=\g$ where
\[
   \g_i^\ga=\bigoplus_{\a(h_\ga)=i}\g_\a,\quad \tilde\g_i^\ga=\bigoplus_{\a(h_\ga)\le i}\g_\a,\quad \tilde\g_i^\ga=\bigoplus_{s\le i}\g_s^\ga.
\]
In the next section we give equivalent definitions to these objects.

Let $\g_\ga=\{ L(z)=\sum\limits_{j=-k}^\infty L_jz^j \ |\ L_j\in\tilde\g_j^\ga,\ j=-k,\ldots,\infty \}$. Let $G_\ga\subset G(\C(z,z^{-1}))$ be the subgroup with the Lie algebra $\g_\ga$, $K_\ga=K\cap G_\ga$, $K_0=\prod\limits_{\ga\in\Gamma} K_\ga$.

\begin{proposition}\label{P:modsp}
$\M^{\chi,\Gamma}_h=\underbrace{K\times\ldots\times K}_{|\Gamma | \text{times}}/K_0$.
\end{proposition}
\begin{proof}
By \refL{Endo} $\prod\limits_{\ga\in\Gamma}G_\ga$ is exactly the stationary subgroup of the given flag distribution in the group of all $G$-distributions. The proposition follows by $K_0=\underbrace{K\times\ldots\times K}_{|\Gamma | \text{times}}\cap \prod\limits_{\ga\in\Gamma}G_\ga$.
\end{proof}
\begin{theorem}\label{T:mspace}
The tangent space to $\M^{\chi,\Gamma}_h$ at the unit consists of elements of the form
\[
    \bigoplus_{\ga\in\Gamma}\sum_{\a\in R^+} a_\a^\ga(z)x_\a,\quad a_\a^\ga(z)\in \C(z)/z^{\a(h_\ga)}\C(z).
\]
\end{theorem}
\begin{proof}
Since $\M^\chi_{h_\ga}=K/K_\ga$, we have $T_e\M^\chi_{h_\ga}={\frak k}(z)/ {\frak k}(z)\cap\g_\ga$ where $\frak k$ is the Lie algebra of the group $K$ considered over $\C$.

Observe that $\g_\ga$ can be characterized as the subalgebra in $\g(z,z^{-1})$ consisting of elements of the form
\[
    L(z)=\sum_{\a\in R,\ i\ge\a(h_\ga)} x_\a z^i.
\]
Indeed, $L_i=\sum\limits_{s\le i} L_i^s$ where $L_i^s\in\g_s$, and $\g_s=\bigoplus\limits_{\a(h_\ga)=s} \g_\a$. Hence the terms $x_\a z^i$, $i\ge\a(h_\ga)$ are absent in the quotient  ${\frak k}/ {\frak k}(z)\cap\g_\ga$. In particular, the whole lower triangle subalgebra is filtered out because it is generated by $x_\a$, $\a\in R_-$, hence only $x_\a z^i$ with $i<0$ could be in the remainder ($i<\a(h_\ga)\le 0$), but there is no negative degrees in ${\frak k}(z)$ (moreover, the whole lower parabolic subalgebra $\tilde\g_0^\ga$ is filtered out for the same reason). Thus we are left with only upper nilpotents, and their exponents exactly give elements in the form claimed in the statement of the theorem.
\end{proof}
\begin{corollary}\label{C:dimm}
\begin{itemize}
\item[$1^\circ$.]
$\dim\M^\chi_{h_\ga}=\sum\limits_{s=1}^k s\dim\g_s^\ga$.
\item[$2^\circ$.]
$\dim\M^{\chi,\Gamma}_h=\sum\limits_{\ga\in\Gamma}\sum\limits_{s=1}^k s\dim\g_s^\ga$ ($\Gamma$ is fixed).
\item[$3^\circ$.]
$\dim\M^\chi_h=\sum\limits_{\ga\in\Gamma}(1+\sum\limits_{s=1}^k s\dim\g_s^\ga)$ ($\Gamma$ is not fixed).
\end{itemize}
where $k$ is the depth of the above defined grading.
\end{corollary}
\begin{proof}
Indeed, for every $\a\in R^+$ the number of moduli is equal to $\a(h_\ga)$, hence $\dim\M^\chi_{h_\ga}=\sum\limits_{\a\in R^+} \a(h_\ga)$. Next, $\g_s^\ga=\bigoplus\limits_{\a(h_\ga)=s}\g_\a$, i.e. $\dim\g_s^\ga=\sharp\{\a\ |\ \a(h_\ga)=s \}$. Since $\dim\g_\a=1$ we obtain the first and second claims. $3^\circ$ follows from $2^\circ$ taking account of elements $\ga\in\Gamma$ which are also counted as moduli.
\end{proof}
{\bf Examples.}

1) For $G=GL(n)$, $V=\C^n$, $h_\ga=diag(d_1,\ldots,d_n)$ \refT{mspace} gives $a_{ij}\in k(z)/z^{d_i-d_j}k(z)$ for the entry $a_{ij}$ of an element of the tangent space. These are the conditions claimed in \cite{Tyur65,Tyur66} for the canonical form of any divisor. Here they have a different meaning.

Assume $h_\ga=diag(1,0,\ldots,0)$ for every $\ga$. Then the corresponding grading is as follows: $\g=\g_{-1}\oplus\g_0\oplus\g_1$ where $\dim\g_{\pm 1}=n-1$ \cite{Sh_UMN_2015}, i.e. according to \refC{dimm}.$3^\circ$ every point contributes $(n-1)+1=n$ parameters into the dimension of the moduli space. If $|\Gamma |=ng$ then the total number of parameters is equal to $n^2g$. Observe that $\Ad G$ is a subgroup of the  group of equivalencies of any matrix divisor because the right multiplication by $g^{-1}$ ($g\in G$) is a multiplication by a global (constant) function while the left multiplication by $g$ is a multiplication by a $K$-valued distribution. Taking quotient by $\Ad G$ kills $n^2-1$ parameters, and we conclude that $\dim\M=n^2(g-1)+1$ which coincides with the known dimension of the moduli space of holomorphic rank $n$ vector bundles on $\Sigma$.

2) Consider $G=SO(2n)$, $V=\C^{2n}$,  $h_\ga=diag(1,0,\ldots,0)$ for every~$\ga$. Then again $\g=\g_{-1}\oplus\g_0\oplus\g_1$ where $\dim\g_{\pm 1}=2n-2$ \cite{Sh_UMN_2015}. Hence every $\ga\in\Gamma$ contributes $2n-1$ into dimension of the moduli space. If $|\Gamma |=ng$ then $\dim\M=(2n-1)ng-\dim G=(2n-1)n(g-1)$ (the subtracting of $\dim G=(2n-1)n$ is due to the action of $\Ad G$ as above).

3) Let $G=Sp(2n)$, $V=\C^{2n}$,  $h_\ga=diag(1,0,\ldots,0)$ for every~$\ga$. Then $\g=\g_{-2}\oplus\g_{-1}\oplus\g_0\oplus\g_1\oplus\g_2$ where $\dim\g_{\pm 1}=2n-2$, $\dim\g_{\pm 2}=1$ \cite{Sh_UMN_2015}. By \refC{dimm}.$2^\circ$ every point $\ga\in\Gamma$ contributes $2\cdot 1+(2n-2)+1=2n+1$ into the dimension of the moduli space. If $|\Gamma |=ng$ then $\dim\M=(2n+1)ng-\dim G=(2n+1)n(g-1)$ (again the subtracting of $\dim G$ is due to the action of $\Ad G$ as above but this time $\dim G=(2n+1)n$).

In all considered cases the dimension of the space of matrix divisors coincides with the dimension of the corresponding moduli space of semi-stable $G$-bundles. To obtain the same result for $SL(n)$ and $G=SO(2n+1)$ we would need consider the matrix divisors with values in the conformal extensions of the corresponding groups (see the Introduction).
\section{Moduli of matrix divisors and Lax operator algebras}\label{S:globa}

Here we will establish an isomorphism of the tangent space at the unit to the moduli space of matrix divisors (with given discrete invariants) with the quotient of the space of $M$-operators by the Lax operator algebra (basically defined by the same invariants). The result goes back to \cite{Klax}. For brevity, we consider only Chevalley groups with a trivial center here. We start with the definition of a Lax operator algebra.

Let $\g$ be a semi-simple Lie algebra over $\C$, $\h$ be its Cartan subalgebra, and $h\in\h$ be such element that  $p_i=\a_i(h)\in\Z_+$ for every simple root $\a_i$ of $\g$. If we denote the root lattice of $\g$ by $Q$ then $h$ belongs to the positive chamber of the dual lattice $Q^*$.

Let $\g_p=\{ X\in\g\ |\ (\ad h)X=pX \}$, and $k=\max\{p\ |\ \g_p\ne 0\}$. Then the decomposition $\g=\bigoplus\limits_{i=-k}^{k}\g_p$ gives a $\Z$-grading on $\g$. For the theory and classification results on such kind of gradings we refer to \cite{Vin}. Call $k$ a \emph{depth} of the grading. Obviously, $\g_p=\bigoplus\limits_{\substack{\a\in R\\ \a(h)=p}}\g_\a$ where $R$ is the root system of $\g$. Define also the following filtration on $\g$:  $\tilde\g_p=\bigoplus\limits_{q=-k}^p\g_q$. Then $\tilde\g_p\subset\tilde\g_{p+1}$ ($p\ge -k$), $\tilde\g_{-k}=\g_{-k},\ldots,\tilde\g_k=\g$, $\tilde\g_p=\g$, $p>k$.

As above, let $\Sigma$ be a complex compact Riemann surface with two given non-intersecting finite sets of marked points: $\Pi$ and $\Gamma$. Assume every $\ga\in\Gamma$ to be assigned with an $h_\ga\in Q^*_+$, and the corresponding grading and filtration. We equip the notation $\g_p$, $\tilde\g_p$ with the upper $\ga$ indicating that the grading (resp. filtration) subspace corresponds to $\ga$. We stress that $\g_p^\ga$, $\tilde\g_p^\ga$ are the same as defined in the previous section. Let $L$ be a meromorphic mapping $\Sigma\to\g$ holomorphic outside the marked points which may have poles of an arbitrary order at the points in $\Pi$, and has the decomposition of the following form at every~$\ga\in\Gamma$:
\begin{equation}\label{E:ga_expan}
   L(z)=\sum\limits_{p=-k}^\infty L_pz^p,\ L_p\in\tilde\g_p^\ga
\end{equation}
where $z$ is a local coordinate in the neighborhood of $\ga$. For simplicity, we assume that the depth of grading $k$ is the same all over $\Gamma$, though it would be no difference otherwise.

We denote by $\L$ a linear space of all such mappings. Since the relation \refE{ga_expan} is preserved under commutator, $\L$ is a Lie algebra called {\it Lax operator algebra}.

Lax operator algebras have emerged in \cite{KSlax} due to the observation by I.Krichever \cite{Klax} that the Lax operators of integrable systems with the spectral parameter on a Riemann surface have a very special Laurent expansions related to the Tyurin parameters of holomorphic vector bundles. In  \cite{Sh_DAN_LOA&gr,Sh_TrGr} they were generalized to the form described here. For the current state of the theory of Lax operator algebras and their applications to integrable systems we refer to \cite{Sh_UMN_2015,Sh_DGr,Sh_DAN_LOA&gr,Sh_TrGr,Sh_TMPh_14,Sh_TrMIAN_15,Sh_MMJ} and references therein.

To give the Lax operator algebra description of the moduli space of matrix divisors introduce the space of $M$-operators, the counterparts of Lax operators in the Lax pairs of integrable systems.

A meromorphic mapping $M:\ \Sigma\to\g$,  holomorphic outside $\Pi$ and $\Gamma$, is called an \emph{$M$-operator} if at any $\ga\in\Gamma$ it has a Laurent expansion
\begin{equation}\label{E:M_oper}
  M(z)=\frac{\nu_\ga h_\ga}{z}+\sum_{i=-k}^\infty M_iz^i
\end{equation}
where $M_i\in\tilde\g_i^\ga$ for $i<0$, $M_i\in\g$ for $i\ge 0$, and $\nu_\ga\in\C$. We denote by $\M_{\Pi,\Gamma,h}$ the space of $M$-operators corresponding to given data $\Pi,\Gamma,h$ (as above, $h=\{ h_\ga\ |\ \ga\in\Gamma\}$).  Obviously, $\L_{\Pi,\Gamma,h}\subset\M_{\Pi,\Gamma,h}$.

Let $(L)_\Pi$ and $(M)_\Pi$ be the restrictions of the corresponding divisors to $\Pi$. For an non-negative divisor $D$ introduce
\[
\begin{aligned}
  \L_{\Pi,\Gamma,h}^D &=\{ L\in\L_{\Pi,\Gamma,h}   \ |\  (L)_\Pi+D\ge 0 \}, \\
  \M_{\Pi,\Gamma,h}^D &=\{ M\in\M_{\Pi,\Gamma,h}   \ |\  (M)_\Pi+D\ge 0 \}.
\end{aligned}
\]

For a Lax operator algebra $\L_{\Pi,\Gamma,h}$, the mapping $\L_{\Pi,\Gamma,h}\to\bigoplus\limits_{\ga\in\Gamma}\g_\ga$ taking any $L\in\L_{\Pi,\Gamma,h}$ to the set of its Laurent expansions at the points in $\Gamma$ is called \emph{localization}.
\begin{theorem}\label{T:quot}
For the tangent space at the unit to the moduli space of matrix divisors we have
\begin{equation}\label{E:tagsp}
 T_e\M^{\chi,\Gamma}_h\cong \M_{\Pi,\Gamma,h}^D/\L_{\Pi,\Gamma,h}^D,
\end{equation}
independently of $\Pi$, $\chi$ and $D$, provided $\deg D<|\Gamma |$. The isomorphism is given by the localization map.
\end{theorem}
\begin{proof}
For $\deg D< |\Gamma|$ the localization map is injective for the reason that an element in the kernel would have $|\Gamma|$ zeroes and no more than $\deg D<|\Gamma|$ poles, hence it is trivial. At any point $\ga\in\Gamma$ the main parts of the Laurent expansions for $L$ and $M$ operators satisfy the same conditions, hence vanish in the quotient $\M_{\Pi,\Gamma,h}/\L_{\Pi,\Gamma,h}$. As for the Tailor parts, for $i=0,\ldots,k-1$ the coefficients at $z^i$ in the quotient run over $\g/{\tilde\g}_i=\g_{i+1}\oplus\ldots\oplus\g_k$ (and vanish for $i\ge k$). Hence $\M_{\Pi,\Gamma,h}^D/\L_{\Pi,\Gamma,h}^D=\bigoplus_{\ga\in\Gamma} \left\{\sum_{i=1}^kL_iz^i\, |\, L_i\in\bigoplus_{i<s}\g_s^\ga\right\}$ where it is exactly the quotient of localizations on the right hand side of the relation. It is similar to the proof of \refC{dimm} that
$\bigoplus_{i<s}\g_s^\ga=\bigoplus\limits_{i<\a(h_\ga)}\g_\a$. Hence $\M_{\Pi,\Gamma,h}^D/\L_{\Pi,\Gamma,h}^D=\bigoplus_{\ga\in\Gamma} \left\{\sum_{i=1}^kL_iz^i\, |\, L_i\in\bigoplus_{i<\a(h_\ga)}\g_\a \right\}$. The last space coincides with $T_e\M^{\chi,\Gamma}_h$ by \refT{mspace}.
\end{proof}
\begin{remark}
The independence of $\chi$ in the theorem is related to the triviality of the centers of the Chevalley groups in question (see remarks in the beinning of the section, and in Introduction).
\end{remark}
\begin{remark}
In general, the right hand side of \refE{tagsp} depends on $D$.  \refT{quot} may be reformulated using the space $\M^{\chi,\Gamma,D}_h$ corresponding to modified equivalence of matrix divisors. Namely, two matrix divisors will be equivalent if their spaces $\Gamma_{\rm gl}^D(\Psi)$ (cf. \refC{nontr}) are non-trivial and equal up to the shift by a constant element of the group.

We note the similarity with \cite[Lemma 5]{Tyur65} which reduces the classification problem for degree $0$ framed bundles to the corresponding problem for effective framed bundles by tensoring the first with two linear bundles of degrees $g-1$ and $1$, respectively. It is the same as to assume the existence of an external (with respect to $\Gamma$) degree $g$ divisor $D$. To stress the similarity we note that for a semisimple group $G$ all $G$-bundles have degree $0$, and that by \refC{nontr} $\Gamma_{\rm gl}^D(\Psi)$ is non-trivial if $\deg D\ge g$.
\end{remark}
\bibliographystyle{amsalpha}


\end{document}